\documentclass{amsart}
\usepackage{amsmath,amssymb,enumerate,mathrsfs,verbatim,graphicx}

\def\dist{\operatorname{dist}}
\def\N{\mathbb N}
\def\eps{\varepsilon}
\def\lin{\operatorname{span}}
\def\conv{\operatorname{co}}
\def\ienv{\operatorname{(I)-env}}
\def\dw{{\delta_w}}
\def\dws{{\delta_{w^*}}}
\def\cl#1#2{\overline{#2}^{#1}}
\def\trn#1{{\left\vert\kern-0.3ex\left\vert\kern-0.3ex\left\vert\hskip 0,1 mm #1\hskip 0,1 mm\right\vert\kern-0.3ex\right\vert\kern-0.3ex\right\vert}}

\newtheorem{theorem}{Theorem}[section]
\newtheorem{proposition}[theorem]{Proposition}
\newtheorem{lemma}[theorem]{Lemma}
\newtheorem{corollary}[theorem]{Corollary}
\theoremstyle{definition}
\newtheorem*{definition}{Definition}

\begin{document}

\title{Quantitative Grothendieck property}
\author{Hana Bendov\'a}

\address{Department of Mathematical Analysis \\
Faculty of Mathematics and Physics\\ Charles University in Prague\\
Sokolovsk\'{a} 83, 186 \ 75\\Praha 8, Czech Republic}

\email{bendova@karlin.mff.cuni.cz}


\thanks{The research was supported by the Grant No. 142213/B-MAT/MFF of the Grant Agency of the Charles University in Prague
and by the Research grant GA\v{C}R P201/12/0290.}

\begin{abstract}
A Banach space $X$ is Grothendieck if the weak and the weak$^*$ convergence of sequences in the dual space $X^*$ coincide. The space $\ell^\infty$ is a classical example of a Grothendieck space due to Grothendieck. We introduce a quantitative version of the Grothendieck property, we prove a quantitative version of the above-mentioned Grothendieck's result and we construct a Grothendieck space which is not quantitatively Grothendieck. We also establish the quantitative Grothendieck property of $L^\infty(\mu)$ for a $\sigma$-finite measure $\mu$.
\end{abstract}

\maketitle


\section{Introduction}
A Banach space $X$ is said to be \emph{Grothendieck} if the weak and the weak$^*$ convergence of sequences in the dual space $X^*$ coincide.
The space $\ell_\infty$ is a classical example of a Grothendieck space due to Grothendieck \cite{Grothendieck}. Some other examples are $C(K)$ where $K$ is an $F$-space \cite{Seever}, weak $L^p$ spaces \cite{Lotz}, and the Hardy space $H^{\infty}$ \cite{Bourgain}. R. Haydon has constructed a Grothendieck space which does not contain $\ell_\infty$ \cite{Haydon}.

In this paper, we introduce a quantitative version of the Grothendieck property. Our inspiration comes from many recent quantitative results. Quite a few properties and theorems have been given a quantitative form lately. Let us mention quantitative versions of Krein's theorem \cite{FHMZ, Granero, GHM, CMR}, quantitative versions of the Eberlein-\v Smulyan and the Gantmacher theorem \cite{AC}, quantitative version of James' compactness theorem \cite{CKS, GHP}, quantitative weak sequential continuity and quantitative Schur property \cite{KPS, KS}, quantification of Dunford Pettis \cite{KK} and reciprocal Dunford-Pettis property \cite{KS2}.

The definition of the Grothendieck property can be rephrased as follows. A~Banach space $X$ is Grothendieck if every weak$^*$ Cauchy sequence in $X^*$ is weakly Cauchy.
The quantitative version is derived from this formulation in the following way.
Let $X$ be a Banach space and $(x^*_n)$ be a bounded sequence in $X^*$. Define
$$ \dw(x^*_n) = \sup_{x^{**}\in B_{X^{**}}} \inf_{n\in\N} \sup_{k,l\geq n} |x^{**}(x^*_k) - x^{**}(x^*_l)|$$
the ``measure of weak non-cauchyness'' of the sequence $(x^*_n)$, and
$$\dws(x^*_n) = \sup_{x\in B_X} \inf_{n\in\N} \sup_{k,l\geq n} |x^*_k(x) - x^*_l(x)|$$
the ``measure of weak$^*$ non-cauchyness'' of the sequence $(x^*_n)$.
The quantities $\dw(x^*_n)$ and $\dws(x^*_n)$ are equal to zero if and only if the sequence $(x^*_n)$ is weakly and weak$^*$ Cauchy, respectively.
We now replace the implication in the definition of the Grothendieck property by an inequality between these two quantities, which is a stronger condition.

\begin{definition}[quantitative Grothendieck property]
Let $c\geq 1$. A Banach space $X$ is \emph{$c$-Grothendieck} if
$$\dw(x^*_n) \leq c \dws(x^*_n)$$
whenever $(x^*_n)$ is a bounded sequence in $X^*$.
\end{definition}
Section 2 establishes the relation between the quantitative Grothendieck property and (I)-envelopes of unit balls. It is then used to prove the following quantitative version of the above-mentioned Grothendieck's result.

\begin{theorem}
\label{q-groth}
The space $\ell_\infty$ is 1-Grothendieck.
\end{theorem}

If $X$ is $c$-Grothendieck for some $c\geq 1$, then it is Grothendieck. In section 3 we show that the converse is not true.
\begin{theorem}
\label{groth!=qgroth}
There is a Grothendieck space which is not $c$-Grothendieck for any $c\geq 1$.
\end{theorem}

Section 4 contains a generalization of Theorem \ref{q-groth} and its consequences.


\section{Relation to (I)-envelopes}

In this section, we characterize the quantitative Grothendieck property using (I)-envelopes. Some results on (I)-envelopes presented in \cite{Kalenda1} and \cite{Kalenda2} have been found extremely useful to us.
\begin{definition}
Let $X$ be a Banach space and $B\subset X^*$. The (I)-envelope of $B$ is defined by
$$\ienv(B)=\bigcap\left\{\cl{\|\cdot\|}{\conv\bigcup_{n=1}^\infty \cl{w^*}{\conv C_n}}:\ B=\bigcup_{n=1}^\infty C_n \right\}.$$
\end{definition}

Any Banach space $X$ is considered to be canonically embedded into its bidual $X^{**}$. If $B$ is a set in a Banach space $X$, then $B$ is regarded as a subset of $X^{**}$ and so is the (I)-envelope of $B$. By $\cl{w^*}{B}$ we mean the weak$^*$ closure of $B$ in $X^{**}$. 

The following lemma, proved by Kalenda \cite[Lemma 2.3]{Kalenda1}, provides the characterization of (I)-envelopes. It allows us to prove Proposition \ref{groth-ienv}, which describes the relation between (I)-envelopes and the quantitative Grothendieck property.

\begin{lemma}
\label{key}
Let $X$ be a Banach space, $B\subset X$ be a closed convex set and $z^{**}\in\cl{w^*}{B}$. Then the following conditions are equivalent:
\begin{enumerate}[(1)]
\item $z^{**}\notin\ienv(B)$;
\item there is a sequence $(\xi^*_n)$ in $B_{X^*}$ such that $$\sup_{x\in B}\limsup_{n\to\infty}{\xi^*_n(x)} < \inf_{n\in\N}{z^{**}(\xi^*_n)};$$
\item there is a sequence $(\xi^*_n)$ in $B_{X^*}$ such that $$\sup_{x\in B}\limsup_{n\to\infty}{\xi^*_n(x)} < \liminf_{n\to\infty}{z^{**}(\xi^*_n)};$$
\item there is a sequence $(\xi^*_n)$ in $B_{X^*}$ such that $$\sup_{x\in B}\limsup_{n\to\infty}{\xi^*_n(x)} < \limsup_{n\to\infty}{z^{**}(\xi^*_n)}.$$
\end{enumerate}
\end{lemma}
 
\begin{proposition}
\label{groth-ienv}
Let $X$ be a Banach space and $c\geq 1$. Then $X$ is $c$-Grothendieck if and only if $\ienv(B_X) \supset \frac1c B_{X^{**}}$.
\end{proposition}
\begin{proof}
Suppose that $X$ is not $c$-Grothendieck. Find a bounded sequence $(x^*_n)$ in $X^*$ such that $\dw(x^*_n)>c\dws(x^*_n)$, i.e.
$$\sup_{x^{**}\in B_{X^{**}}}\inf_{n\in\N}\sup_{k,l\geq n}{|x^{**}(x^*_k)-x^{**}(x^*_l)|} > c \sup_{x\in B_X}\inf_{n\in\N}\sup_{k,l\geq n}{|x^*_k(x)-x^*_l(x)|}.$$
There is no loss of generality in assuming that $x^*_n \in B_{X^*}$, $n\in\N$. Let $x^{**}\in B_{X^{**}}$ be such that
$$\inf_{n\in\N}\sup_{k,l\geq n}{|x^{**}(x^*_k)-x^{**}(x^*_l)|} > c \sup_{x\in B_X}\inf_{n\in\N}\sup_{k,l\geq n}{|x^*_k(x)-x^*_l(x)|},$$
and set $z^{**}=\frac1c x^{**}$. Then $z^{**}\in\frac1c B_{X^{**}}$, and
\begin{equation}
\label{limsup-liminf}
\begin{aligned}
\limsup_{n\to\infty}{z^{**}(x^*_n)}-\liminf_{n\to\infty}{z^{**}(x^*_n)} &= \inf_{n\in\N}\sup_{k,l\geq n}{|z^{**}(x^*_k)-z^{**}(x^*_l)|} \\
&= \frac1c \inf_{n\in\N}\sup_{k,l\geq n}{|x^{**}(x^*_k)-x^{**}(x^*_l)|}\\
&> \sup_{x\in B_X}\inf_{n\in\N}\sup_{k,l\geq n}{|x^*_k(x)-x^*_l(x)|} \\
&= \sup_{x\in B_X} \left(\limsup_{n\to\infty}{x^*_n(x)}-\liminf_{n\to\infty}{x^*_n(x)} \right).
\end{aligned}
\end{equation}
Find subsequences $(y^*_k)$ and $(z^*_k)$ of the sequence $(x^*_n)$ for which\linebreak $\limsup_{n\to\infty}{z^{**}(x^*_n)} = \lim_{k\to\infty}{z^{**}(y^*_k)}$, and
$\liminf_{n\to\infty}{z^{**}(x^*_n)} = \lim_{k\to\infty}{z^{**}(z^*_k)}$.
Set $\xi^*_k=\frac12(y^*_k-z^*_k)$, $k\in\N$. Then $(\xi^*_k)$ is a sequence in $B_{X^*}$, and
$$
\begin{aligned}
\lim_{k\to\infty}{z^{**}(\xi^*_k)}&=\frac12\left(\lim_{k\to\infty}{z^{**}(y_k)}-\lim_{k\to\infty}{z^{**}(z_k)}\right) \\
&=\frac12\left(\limsup_{n\to\infty}{z^{**}(x^*_n)} - \liminf_{n\to\infty}{z^{**}(x^*_n)}\right) \\
&\stackrel{(\ref{limsup-liminf})}{>} \frac12 \sup_{x\in B_X} \left(\limsup_{n\to\infty}{x^*_n(x)}-\liminf_{n\to\infty}{x^*_n(x)} \right) \\
&\geq \frac12 \sup_{x\in B_X} \left(\limsup_{k\to\infty}{y^*_k(x)}-\liminf_{k\to\infty}{z^*_k(x)} \right) \\
&\geq \frac12 \sup_{x\in B_X} \limsup_{k\to\infty}{(y^*_k(x)-z^*_k(x))} \\
&= \sup_{x\in B_X} \limsup_{k\to\infty}{\xi^*_k(x)}.
\end{aligned}
$$
By Lemma \ref{key}, $z^{**}\notin \ienv(B_X)$, and so $\ienv(B_X)\not\supset \frac1c B_{X^{**}}$.

Now suppose that $X$ is $c$-Grothendieck and fix arbitrary $z^{**}\in\frac1c B_{X^{**}}$. Let $(x^*_n)$ be a sequence in $B_{X^*}$. Then $\dw(x^*_n)\leq c \dws(x^*_n)$, that is
$$\sup_{x^{**}\in B_{X^{**}}} \hskip -1.2 mm \Big(\limsup_{n\to\infty}{x^{**}(x^*_n)}-\liminf_{n\to\infty}{x^{**}(x^*_n)} \Big) \leq c\sup_{x\in B_X} \hskip -0.8 mm\Big(\limsup_{n\to\infty}{x^*_n(x)}-\liminf_{n\to\infty}{x^*_n(x)} \Big).$$
Since $cz^{**}\in B_{X^{**}}$, it follows that
$$
\begin{aligned}
\limsup_{n\to\infty}{z^{**}(x^*_n)} - \liminf_{n\to\infty}{z^{**}(x^*_n)} &= \frac1c\Big(\limsup_{n\to\infty}{cz^{**}(x^*_n)} - \liminf_{n\to\infty}{cz^{**}(x^*_n)}\Big) \\
&\leq \sup_{x\in B_X} \Big(\limsup_{n\to\infty}{x^*_n(x)}-\liminf_{n\to\infty}{x^*_n(x)} \Big).
\end{aligned}
$$
For $k\in\N$ find an $x_k\in B_X$ satisfying
$$\limsup_{n\to\infty}{x^*_n(x_k)}-\liminf_{n\to\infty}{x^*_n(x_k)} > \limsup_{n\to\infty}{z^{**}(x^*_n)} - \liminf_{n\to\infty}{z^{**}(x^*_n)} - \frac2k.$$
Then either $\limsup_{n\to\infty}{x^*_n(x_k)} > \limsup_{n\to\infty}{z^{**}(x^*_n)} - \frac1k$ or $\liminf_{n\to\infty}{x^*_n(x_k)} < \liminf_{n\to\infty}{z^{**}(x^*_n)} + \frac1k$.
If the former inequality holds for infinitely many $k\in\N$, then $\limsup_{n\to\infty}{z^{**}(x^*_n)} \leq \sup_{x\in B_X} \limsup_{n\to\infty}{x^*_n(x)}$. Otherwise the latter holds for infinitely many $k\in\N$, and $\liminf_{n\to\infty}{z^{**}(x^*_n)} \geq \inf_{x\in B_X}\liminf_{n\to\infty}{x^*_n(x)}$, which gives
$\limsup_{n\to\infty}{-z^{**}(x^*_n)} \leq \sup_{x\in B_X} \limsup_{n\to\infty}{x^*_n(x)}$.

So far we have shown that whenever $(x^*_n)$ is a sequence in $B_{X^*}$, either
$$\limsup_{n\to\infty}{z^{**}(x^*_n)} \leq \sup_{x\in B_X} \limsup_{n\to\infty}{x^*_n(x)} \quad \text{or} \quad \limsup_{n\to\infty}{-z^{**}(x^*_n)} \leq \sup_{x\in B_X} \limsup_{n\to\infty}{x^*_n(x)}.$$
Consider now an arbitrary sequence $(x^*_n)$ in $B_{X^*}$.
Set $(y^*_n)_n\hskip -0.7 mm =\hskip -0.3 mm(x^*_1,-x^*_1,x^*_2,-x^*_2,\dots\hskip -0.05 mm)$. From what has already been proved, we obtain
$$\limsup_{n\to\infty}{z^{**}(y^*_n)} = \limsup_{n\to\infty}{-z^{**}(y^*_n)} \leq \sup_{x\in B_X} \limsup_{n\to\infty}{y^*_n(x)}.$$
Hence
$$
\begin{aligned}
\limsup_{n\to\infty}{z^{**}(x^*_n)} &\leq \limsup_{n\to\infty}{z^{**}(y^*_n)} \leq \sup_{x\in B_X} \limsup_{n\to\infty}{y^*_n(x)} \\
&= \sup_{x\in B_X} \max\{\limsup_{n\to\infty}{x^*_n(x)},\limsup_{n\to\infty}{-x^*_n(x)}\} \\
&= \sup_{x\in B_X} \limsup_{n\to\infty}{x^*_n(x)}.
\end{aligned}$$
Lemma \ref{key} gives $z^{**}\in\ienv(B)$, which shows that $\frac1c B_{X^{**}}\subset \ienv(B)$.
\end{proof}

We are now able to prove Theorem \ref{q-groth}. It is a trivial consequence of Proposition \ref{groth-ienv} and Kalenda's theorem \cite[Example 4.1]{Kalenda1}, which says that $\ienv(B_{\ell_\infty})=B_{(\ell_\infty)^{**}}$.


\section{The relation between Grothendieck property and its quantitative version}

We have already mentioned that the quantitative Grothendieck property is\linebreak stronger than its original qualitative version.
This section is devoted to the construction of a Banach space which is Grothendieck but not $c$-Grothendieck for any $c\geq 1$.

The following proposition is a strengthening of Kalenda's theorem \cite[Theorem 2.2]{Kalenda2}, and its proof is a modification of the original one.

\begin{proposition}
\label{renorm}
Let $X$ be a nonreflexive Banach space and $c\geq 1$. Then there exists an equivalent norm $\trn{\cdot}$ on $X$ such that $(X,\trn{\cdot})$ is not $c$-Grothendieck.
\end{proposition}
\begin{proof}
If $X$ is separable, then $\ienv(B_X)=B_X$ (see \cite[Remark 1.1(ii)]{Kalenda1}). By nonreflexivity, $\frac1cB_{X^{**}}\not\subset B_X$ for any $c\geq1$, so the assertion follows from Proposition~\ref{groth-ienv}. Renorming is not necessary.

Suppose that $X$ is nonseparable. Find a separable subspace $Y\subset X$ which is not reflexive. Let $x^*\in S_{X^*}$ be such that $x^*|_Y=0$, and fix $x_0\in X$ with $x^*(x_0)=1$. Obviously, $\|x_0\|\geq 1$. The bidual $Y^{**}$ can be canonically identified with the $w^*\text{-}$closure of $Y$ in $X^{**}$, and $Y=Y^{**}\cap X$. Thus we can find some $y^{**}\in S_{Y^{**}}\setminus X$. Set $Z=\lin(Y\cup\{x_0\})$. Since $y^{**}\in Z^{**}\setminus Z$, $y^{**}|_{B_{Z^*}}$ is not weak$^*$ continuous. Clearly, $Z$ is separable, thus $(B_{Z^*},w^*)$ is metrizable, hence $y^{**}|_{B_{Z^*}}$ is not even weak$^*$ sequentially continuous. Therefore there exists a sequence $(\widetilde{x^*_n})$ in $B_{Z^*}$ weak$^*$ converging to $0$ and $\eta\in(0,1]$ such that $y^{**}(\widetilde{x^*_n})\geq \eta$, $n\in\N$. For each $n\in\N$ extend $\widetilde{x^*_n}$ to $x^*_n\in B_{X^*}$ by the Hahn-Banach theorem.

Define
$$B=\left\{x\in X:\ \|x-x^*(x)x_0\| \leq 1 \text{ and } |x^*(x)| + \dist(x-x^*(x)x_0,Y)\leq\frac{\eta}{c}\right\}.$$
Then $B$ is a closed absolutely convex set. Moreover, we show that
$$\frac{\eta}{c(2+\|x_0\|)}B_X \subset B \subset \left(1+\frac{\eta}{c}\right)\|x_0\|B_X.$$
For $x\in B$ we have $$\|x\| \leq 1 + |x^*(x)|\|x_0\| \leq \|x_0\| + \frac{\eta}{c}\|x_0\|,$$ which proves the second inclusion.
To prove the first one let $x\in B_X$. Then
$$|x^*(x)|\leq 1,$$
$$\|x-x^*(x)x_0\| \leq 1 + \|x_0\|,$$
$$\dist(x-x^*(x)x_0,Y) \leq \|x-x^*(x)x_0\| \leq 1 + \|x_0\|.$$
Hence for $z=\frac{\eta x}{c(2+\|x_0\|)}$ we have
$$\left\|z - x^*\left(z\right)x_0\right\| = \frac{\eta}{c(2+\|x_0\|)}\|x-x^*(x)x_0\| \leq \frac{\eta}{c}\frac{1+\|x_0\|}{2+\|x_0\|}\leq 1,$$
and
$$
\begin{aligned}
\left|x^*(z)\right| + \dist\left(z-x^*\left(z\right)x_0,Y\right) &\leq \frac{\eta}{c(2+\|x_0\|)} + \frac{\eta}{c(2+\|x_0\|)}(1+\|x_0\|) \\
&\leq \frac{\eta}{c}\frac{1+1+\|x_0\|}{2+\|x_0\|} = \frac{\eta}{c}.
\end{aligned}
$$
Thus $B$ is the unit ball of an equivalent norm on $X$. According to Proposition~\ref{groth-ienv}, we shall have established the proposition if we show that $\frac1c\cl{w^*}{B}\not\subset\ienv(B)$.

Set $z^{**}=\frac1c(\frac{\eta}{c}x_0+y^{**})$. Let $(y_\nu)$ be a net in $B_Y$ weak$^*$ converging to $y^{**}$. Then $\frac{\eta}{c}x_0 + y_\nu$ weak$^*$ converges to $\frac{\eta}{c}x_0+y^{**}$. Furthermore, $\frac{\eta}{c}x_0 + y_\nu\in B$ since
$$x^*\left(\frac{\eta}{c}x_0 + y_\nu\right) = \frac{\eta}{c} x^*(x_0) + x^*(y_\nu)=\frac{\eta}{c},$$
$$\left\|\frac{\eta}{c}x_0 + y_\nu - x^*\left(\frac{\eta}{c}x_0 + y_\nu\right)x_0 \right\| = \|y_\nu\| \leq 1,$$
$$\dist\left(\frac{\eta}{c}x_0 + y_\nu - x^*\left(\frac{\eta}{c}x_0 + y_\nu\right)x_0, Y\right) = \dist(y_\nu, Y) = 0.$$
Therefore $z^{**}\in\frac1c\cl{w^*}{B}$. It remains to prove that $z^{**}\notin\ienv(B)$. Define $\xi^*_n=x^*+x^*_n$, $n\in\N$. Then $(\xi^*_n)$ is a bounded sequence in $X^*$, and
$$
\begin{aligned}
\liminf_{n\to\infty} z^{**}(\xi^*_n) &= \frac1c\liminf_{n\to\infty}\left(x^*\left(\frac{\eta}{c}x_0\right) + x^*_n\left(\frac{\eta}{c}x_0\right) + y^{**}(x^*) + y^{**}(x^*_n)\right) \\
&=\frac1c\left(\frac{\eta}{c} + \frac{\eta}{c}\lim_{n\to\infty}x^*_n(x_0) + y^{**}(x^*) + \liminf_{n\to\infty}y^{**}(x^*_n) \right) \\
&\geq\frac1c\left(\frac{\eta}{c} + 0 + 0 + \eta\right) = \frac{c+1}{c}\frac{\eta}{c} > \frac{\eta}{c}. \\
\end{aligned}
$$
In the last inequality, we have used the following two facts. Firstly, $x^*_n(x_0)\to 0$, as $x_0\in Z$. Secondly, $y^{**}(x^*)=0$, since $y^{**}\in \cl{w^*}{Y}$ and $x^*|_Y=0$.
On the other hand, if $x\in B$, $y\in Y$ are arbitrary, then
$$
\begin{aligned}
\hskip -0.5 mm\limsup_{n\to\infty}\xi^*_n(x) &\hskip -0.3 mm = \hskip -0.3 mm x^*(x) \hskip -0.3 mm + \hskip -0.3 mm \limsup_{n\to\infty}(x^*_n(x-x^*(x)x_0-y) + x^*_n(x_0)x^*(x) + x^*_n(y)) \\
&\hskip -0.3 mm = \hskip -0.3 mm x^*(x)\hskip -0.3 mm + \hskip -0.3 mm \limsup_{n\to\infty}x^*_n(x-x^*(x)x_0-y)\hskip -0.3 mm +\hskip -1.3 mm \lim_{n\to\infty}x^*_n(x_0)x^*(x)\hskip -0.3 mm +\hskip -1.4 mm \lim_{n\to\infty}x^*_n(y) \\
&\hskip -0.3 mm \leq \hskip -0.3 mm x^*(x) \hskip -0.3 mm + \hskip -0.3 mm \limsup_{n\to\infty}\|x^*_n\|\|x-x^*(x)x_0-y\| + 0 + 0 \\
&\hskip -0.3 mm \leq \hskip -0.3 mm x^*(x) \hskip -0.3 mm + \hskip -0.3 mm \limsup_{n\to\infty}\|x-x^*(x)x_0-y\|,
\end{aligned}
$$
because $x_0,y\in Z$, and $x^*_n(z)\to 0$ for all $z\in Z$. Hence for every $x\in B$
$$\limsup_{n\to\infty}\xi^*_n(x) \leq |x^*(x)| + \dist(x-x^*(x)x_0, Y) \leq \frac{\eta}{c}.$$
We thus obtain
$$\liminf_{n\to\infty} z^{**}(\xi^*_n)>\frac{\eta}{c}\geq \sup_{x\in B}\limsup_{n\to\infty}\xi^*_n(x).$$
Lemma \ref{key} yields $z^{**}\notin\ienv(B)$, which completes the proof.
\end{proof}

\begin{lemma}
\label{sum}
Suppose that $X_n$, $n\in\N$, are Grothendieck spaces. Then the space $X=\oplus_{\ell_2}X_n$ is also Grothendieck.
\end{lemma}
\begin{proof}
The dual space $X^*$ and the bidual space $X^{**}$ can be represented as $\oplus_{\ell_2}X^*_n$ and $\oplus_{\ell_2}X^{**}_n$, respectively. Let $(x^*_k)$ be a sequence in $X^*$ which weak$^*$ converges to $x^*\in X^*$.
For $x\in X$ we have $x^*_k(x) \to x^*(x)$, that is
$$\sum_{n=1}^{\infty}{x^*_k(n)(x(n))} \to \sum_{n=1}^{\infty}{x^*(n)(x(n))},\ k\to\infty.$$
Let $n\in\N$. If $x_n\in X_n$, then $\bar{x}_n = (
0,\dots,0,x_n,0,0,\dots)\in X$, and so 
$$x^*_k(n)(x_n) = x^*_k(\bar{x}_n) \to x^*(\bar{x}_n) = x^*(n)(x_n),\ k\to\infty.$$
Hence the sequence $(x^*_k(n))_k$ converges to $x^*(n)$ in the weak$^*$ topology, and by the Grothendieck property even in the weak topology. 

To prove that $x^*_k$ weakly converges to $x^*$, fix arbitrary $x^{**}\in X^{**}$. Then $x^{**}(n)(x^*_k(n))\to x^{**}(n)(x^*(n))$, $n\in\N$. We need to establish
$$\lim_{k\to\infty} \sum_{n=1}^{\infty} x^{**}(n)(x^*_k(n)) = \lim_{k\to\infty} x^{**}(x^*_k) = x^{**}(x^*) = \sum_{n=1}^{\infty} x^{**}(n)(x^*(n)),$$
so the proof is completed by showing that the sum $\sum_{n=1}^{\infty} x^{**}(n)(x^*_k(n))$ is uniformly convergent with respect to $k\in\N$.

Let $\eps>0$ and $k\in\N$ be arbitrary. If $j\in\N$, then
$$
\begin{aligned}
\Big|\sum_{n=j}^\infty x^{**}(n)(x^*_k(n)) \Big| &\leq \sum_{n=j}^\infty \|x^{**}(n)\| \|x^*_k(n)\| \\
&\leq \Big(\sum_{n=j}^\infty \|x^{**}(n)\|^2\Big)^{\frac12} \Big(\sum_{n=j}^\infty \|x^*_k(n)\|^2\Big)^{\frac12} \\
&\leq \Big(\sum_{n=j}^\infty \|x^{**}(n)\|^2\Big)^{\frac12} \|x^*_k\|_{X^*}.
\end{aligned}
$$
The sequence $(x^*_k)_k$ is bounded by the uniform boundedness principle. Hence $M>\nolinebreak 0$ can be found such that $\|x^*_k\|_{X^*} \leq M$, $k\in\N$. As $x^{**}\in \oplus_{\ell_2}X^{**}_n$, the sum $\sum_{n=1}^\infty \|x^{**}(n)\|^2$ is convergent. Thus we can choose $j_0\in\N$ such that for $j\geq j_0$
$$\sum_{n=j}^\infty \|x^{**}(n)\|^2 \leq \frac{\eps^2}{M^2}.$$
Then for all $j\geq j_0$
$$ \Big|\sum_{n=j}^\infty x^{**}(n)(x^*_k(n))\Big| \leq \left(\frac{\eps^2}{M^2}\right)^{\frac12}\cdot M = \eps,$$
which is the desired conclusion.
\end{proof}

\begin{lemma}
\label{subspace}
Let $X$ be a Banach space and $c\geq 1$. If $X$ is $c$-Grothendieck, and $Y$ is a quotient of $X$, then $Y$ is $c$-Grothendieck.
\end{lemma}
\begin{proof}
Let $q\colon X\to Y$ be a quotient map. It is easily seen that the dual operator $q^*\colon Y^*\to X^*$ is an isometric embedding. Consequently, $q^{**}\colon X^{**}\to Y^{**}$ satisfy $q^{**}(B_{X^{**}})=B_{Y^{**}}$. Indeed, for $x^{**}\in B_{X^{**}}$
$$\|q^{**}x^{**}\| = \|x^{**}\circ q^*\| \leq \|x^{**}\|\|q^*\|= \|x^{**}\| \leq 1,$$
thus $q^{**}x^{**} \in B_{Y^{**}}$. Let $y^{**}\in B_{Y^{**}}$ be arbitrary. Define a linear functional $x^{**}$ on $q^*(Y^*)\subset X^*$ by $x^{**}(q^*y^*) = y^{**}(y^*)$, $y^*\in Y^*$,
and extend it to a linear functional on $X^*$ with the same norm by the Hahn-Banach theorem. Obviously, $\|x^{**}\|=\|y^{**}\|$ and $q^{**}x^{**}=y^{**}$.

Let $(y^*_n)$ be a bounded sequence in $Y^*$. Then
\begin{equation}
\label{weakX-Y}
\begin{aligned}
\dw(q^*y^*_n) &= \sup_{x^{**}\in B_{X^{**}}} \inf_{n\in\N} \sup_{k,l\geq n} |x^{**}(q^*y^*_k) - x^{**}(q^*y^*_l)| \\
&= \sup_{x^{**}\in B_{X^{**}}} \inf_{n\in\N} \sup_{k,l\geq n} |q^{**}x^{**}(y^*_k) - q^{**}x^{**}(y^*_l)| \\
&= \sup_{y^{**}\in B_{Y^{**}}} \inf_{n\in\N} \sup_{k,l\geq n} |y^{**}(y^*_k) - y^{**}(y^*_l)| = \dw(y^*_n), \\
\end{aligned}
\end{equation}
and
\begin{equation}
\label{weaksX-Y}
\begin{aligned}
\dws(q^*y^*_n) &= \sup_{x\in B_X} \inf_{n\in\N} \sup_{k,l\geq n} |q^*y^*_k(x) - q^*y^*_l(x)| \\
&= \sup_{x\in B_X} \inf_{n\in\N} \sup_{k,l\geq n} |y^*_k(qx) - y^*_l(qx)| \\
&= \sup_{x\in X,\|x\|<1} \inf_{n\in\N} \sup_{k,l\geq n} |y^*_k(qx) - y^*_l(qx)| \\
&= \sup_{y\in y,\|y\|<1} \inf_{n\in\N} \sup_{k,l\geq n} |y^*_k(y) - y^*_l(y)| \\
&= \sup_{y\in B_Y} \inf_{n\in\N} \sup_{k,l\geq n} |y^*_k(y) - y^*_l(y)| = \dws(y^*_n),
\end{aligned}
\end{equation}
where the fourth equality follows from the fact that $q$ is a quotient map.
Since $X$ is $c$-Grothendieck, $\dw(q^*y^*_n) \leq c \dws(q^*y^*_n)$. Together with (\ref{weakX-Y}) and (\ref{weaksX-Y}), it yields
$\dw(y^*_n) \leq c\dws(y^*_n)$, so $Y$ is $c$-Grothendieck.
\end{proof}

\begin{proof}[Proof of Theorem \ref{groth!=qgroth}]
By Proposition \ref{renorm}, for each $n\in\N$ we can find an equivalent norm $\|\cdot\|_n$ on $\ell_\infty$ such that the space $X_n=(\ell_\infty,\|\cdot\|_n)$ is not $n$-Grothendieck. Set $X=\oplus_{\ell_2} X_n$. Then $X$ is Grothendieck by Lemma \ref{sum}, for all $X_n$, $n\in\N$, are Grothendieck spaces. Moreover, each $X_n$ is a quotient of $X$.
Suppose that there is some $c\geq 1$ such that $X$ is $c$-Grothendieck. Find $n\in\N$, $n>c$. Then $X$ is $n$-Grothendieck and, by Lemma \ref{subspace}, $X_n$ should also be $n$-Grothendieck, which is a~contradiction.
\end{proof}


\section{More general results}

Kalenda's theorem \cite[Example 4.1]{Kalenda1}, which we have used to prove a quantitative version of Grothendieck's theorem (Theorem \ref{q-groth}), can be generalized and then applied in the same way to obtain more general quantitative results.

\begin{theorem}
\label{gen-kal}
Let $\Gamma$ be a set and $E=\ell_{\infty}(\Gamma)$. Then $\ienv{B_E}=B_{E^{**}}$.
\end{theorem}
\begin{proof}
The proof of \cite[Example 4.1]{Kalenda1} works here as well. It suffices to replace $\N$ by $\Gamma$ in the right places. Lemmata 4.4 and 4.5 remain unchanged. We prove Lemma 4.6 for sequences of measures on $\Gamma$ by substituting $\N$ with $\Gamma$ wisely. Then we use it to prove Propositions 4.3 and 4.2 for measures on $\Gamma$ just as in the original proof.
\end{proof}

\begin{theorem}
\label{gen-q-groth}
The space $\ell_{\infty}(\Gamma)$ is $1$-Grothendieck for each set $\Gamma$.
\end{theorem}
\begin{proof}
It follows from Theorem \ref{gen-kal} and Proposition \ref{groth-ienv}.
\end{proof}

\begin{corollary}
\label{Linftymu}
Let $\mu$ be a $\sigma$-finite measure on a measurable space $X$. Then $L^\infty(\mu)$ is 1-Grothendieck.
\end{corollary}
\begin{proof}
If $\mu$ is $\sigma$-finite, then $L^{\infty}(\mu)$ is $1\text{-}$injective (see for instance \cite[(5.91)]{FHHMZ}) and thus $1\text{-}$complemented in $\ell_\infty(\Gamma)$ for some set $\Gamma$, which is a $1\text{-}$Grothendieck space by Theorem \ref{gen-q-groth}. By Lemma \ref{subspace}, $L^\infty(\mu)$ is $1\text{-}$Grothendieck.
\end{proof}

The space in Corollary \ref{Linftymu} is $1\text{-}$complemented in $\ell_\infty(\Gamma)$.
In fact, not only $1\text{-}$complemented subspaces but all quotients of $\ell_\infty(\Gamma)$ are $1\text{-}$Grothendieck.

\begin{corollary}
Let $\Gamma$ be an arbitrary set. Each quotient of the space $\ell_\infty(\Gamma)$ is $1\text{-}$Grothendieck.
\end{corollary}
\begin{proof}
It is a consequence of Theorem \ref{gen-q-groth} and Lemma \ref{subspace}.
\end{proof}

Let us remark finally that we do not know whether the other spaces with the Grothendieck property mentioned in the introduction enjoy the quantitative version as well.

\subsection*{Acknowledgement} The author would like to thank Ond\v rej Kalenda for suggesting this research topic and for valuable discussions and comments.

\bibliography{qgrothendieck}
\bibliographystyle{plain}

\end{document}